\documentclass[10pt,twoside,a4]{amsart}
\usepackage{amssymb,amsmath,amsfonts,amsthm,graphicx,amscd}
\usepackage{enumerate,mathrsfs,longtable,bm,float,xcolor,tikz,tikz-cd}
\usepackage[dvips]{epsfig}
\usepackage{amsmath,amsfonts,amssymb,amsthm,bm,float,makecell,booktabs,MnSymbol}

\usepackage{enumerate}
\usepackage{amscd}

\usetikzlibrary{positioning}
\usepackage[all,cmtip,line]{xy}
\usepackage[normalem]{ulem}

\DeclareMathAlphabet{\mathpzc}{OT1}{pzc}{m}{it}

\numberwithin{equation}{section}

\theoremstyle{plain}
\newtheorem{lem}{Lemma}[section]

\newtheorem{thm}[lem]{Theorem}
\newtheorem{prop}[lem]{Proposition}

\newtheorem{cor}[lem]{Corollary}
\theoremstyle{definition}
\newtheorem{exa}[lem]{Example}
\newtheorem{rem}[lem]{Remark}
\newtheorem{defn}[lem]{Definition}
\newtheorem{definition}[lem]{Definition}

\begin{document}

	\baselineskip 13truept

	\title{On $S$-Noetherian Lattices}

	\author{Sachin Sarode*,  Chetan Patil** and Vinayak Joshi***}
	\address{\rm *Department of Mathematics, Shri Muktanand College\\ Gangapur, Dist. Chh. Sambhajinagar - 431 109, India.} \email{sarodemaths@gmail.com}

	\address{\rm **SVKM NMIMS Global University, School of Technology Management and Engineering,
		Dhule-424 001, India.}
	\email{patilcs19@gmail.com}
	\address{\rm ***Department of Mathematics, Savitribai Phule Pune University, Pune-411 007, India.}
	\email{vvjoshi@unipune.ac.in \\
		vinayakjoshi111@yahoo.com }


	\maketitle

	\begin{abstract}
 In this paper, we define and study $S$-Noetherian lattices as a natural generalization of Noetherian rings. We prove that a ring $R$ is $S$-Noetherian if and only if its ideal lattice, $Id(R)$, is $S_L$-Noetherian. Furthermore, we establish a Cohen–Kaplansky type theorem for $S$-Noetherian lattices, showing that $L$ is $S$-Noetherian if and only if every $S$-prime element of $L$ is $S$-compact. Finally, we introduce the concept of $S$-primary elements - a generalization of primary elements in multiplicative lattices and demonstrate the existence (when every element of $S$ is meet principal, as is the case for the ideal lattice of a ring) and uniqueness of $S$-primary decomposition in $S$-Noetherian lattices.
	\end{abstract}

	\vskip 5truept

	\vskip 5truept


	\baselineskip 14truept
	\section{Introduction}\label{intro}

The study of Noetherian rings remains a cornerstone of commutative algebra, primarily due to  the existence of primary decomposition. In recent years, significant research has focused on generalizing these classical notions to various $S$-variants. A particularly fruitful direction has been the introduction of the $S$-Noetherian property. Originally introduced by D. D. Anderson and T. Dumitrescu \cite{AT}, where they established \(S\)-variants of classical results, including Cohen's theorem and the Eakin-Nagata theorem, and examined the property in polynomial and power series extensions.  Hamed and Hizem \cite{HH} and Bilgin, Reyes and Tekir \cite{BRT} studied the corresponding notion for modules and its characterizations by chain conditions.

In this paper, we extend this theory to  multiplicative lattices, providing a complete abstraction of $S$-Noetherian rings.  Multiplicative lattices, introduced by Ward and Dilworth \cite{WD}, serve as an ideal framework for this unification,  as they allow for the study of ideal-theoretic properties independent of the underlying element-wise ring structure.    One of our central results, an analogue of the Cohen-Kaplansky type theorem,  is a characterization theorem which states that an $r$-lattice $L$ is $S$-Noetherian if and only if every prime element is $S$-compact, and we show that every $S$-Noetherian lattice satisfies the $S$-stationary property (the converse holds when $S$ is countable. In this context, the $S$-stationary property serves as the natural lattice-theoretic $S$-variant of the $ACC$ in rings, while $S$-compact elements are an abstraction of $S$-finite ideals.

Massaoud \cite{M} and Visweswaran \cite{V} independently introduced the concept of $S$-primary ideals in commutative rings as a generalization of classical primary ideals. Subsequently, T. Singh, A. Ansari, and S. Kumar \cite{SAK} established the existence of $S$-primary decomposition for $S$-Noetherian rings.

In the present work, we extend these results to the lattice-theoretic setting. We introduce the concept of $S$-primary elements in multiplicative lattices and establish both the existence (under the assumption that every element of $S$ is meet principal) and uniqueness of $S$-primary decomposition within an $S$-Noetherian lattice. Finally, we bridge the gap between ring-theoretic results and our lattice-theoretic abstractions. We establish a rigorous correspondence: if $R$ is a commutative ring with identity and $S$ is a multiplicatively closed subset of $R$, then an $S$-primary ideal of $R$ corresponds precisely to an $S_L$-primary element of the ideal lattice $L = Id(R)$, where $S_L$ is the induced multiplicatively closed subset of $L$. This correspondence ensures that our results for $S$-Noetherian lattices directly recover, unify, and extend existing theorems in commutative algebra.

	We begin with the definitions and terminology.

		\begin{defn}
		A \textit{multiplicative lattice} is a complete lattice $L$ together with a binary operation (multiplication) ``$\cdot$" satisfying:
		\begin{enumerate}
			\item $(L,\;\cdot)$ is commutative and associative,
			\item multiplication distributes over arbitrary joins,
			\item $1$ is the multiplicative identity.
		\end{enumerate}
	\end{defn}

	If $L$ is a multiplicative lattice and $a, b \in L$, then $a\cdot b \leq a \wedge b$. This property follows from the fact that the multiplication distributes over joins and 1 is the multiplicative identity.

	It is well known that the ideal lattice of a commutative ring with unity is a multiplicative lattice; however, not every multiplicative lattice is  the ideal lattice of a commutative ring with unity. For instance, consider the multiplicative lattice $L$ described in  \cite[Example 2.4]{JS}. Since $L$ is non-modular, it cannot be the ideal lattice of any ring, because the ideal lattice of a ring is necessarily modular.

	\begin{defn}[Sarode, Patil and Joshi \cite{SJP}]\label{def:vlattice}
		A complete lattice $L$ is said to be a \textit{$V$-lattice} if there exists a commutative, associative binary operation ``$\cdot$" on $L$ such that:
		\begin{enumerate}
		\item $a \leq b \implies a \cdot c \leq b \cdot c$, for  $a, b, c\in L$.
			\item $a \cdot b \leq a \wedge b$, for  $a, b\in L$.
			\item $a \cdot 1 = a$ for every $a\in L$.
		\end{enumerate}
	\end{defn}

	\begin{rem} \label{1.7.}
		Every multiplicative lattice is a $V$-lattice, but the converse may fail to hold. For instance, let $L$ be the pentagon $N_5=\{0,a,b,c,1\}$ with $0<a<b<1$ and $0<c<1$, where $c$ is incomparable with $a$ and $b$, and define $x\cdot y=x\wedge y$. This operation is commutative, associative and monotone, $x\cdot y\leq x\wedge y$, and $x\cdot 1=x$, so $L$ is a $V$-lattice. However, $L$ is not a multiplicative lattice, since multiplication does not distribute over joins: $b\cdot(a\vee c)=b\wedge 1=b$, whereas $(b\cdot a)\vee(b\cdot c)=a\vee 0=a$.
	\end{rem}

		\begin{defn}
			A nonempty subset $S$ of $L_{*}$ in a $V$-lattice $L$ is \textit{multiplicatively closed} if $s_{1} \cdot s_{2} \in S$ whenever $s_{1}, s_{2} \in S$. Here    $L_{*}$ denotes the set of all compact elements of a complete lattice $ L $. By a \textit{compact} element $c$ of a complete lattice $ L $, we mean
			if $ c \leq \bigvee_{\alpha} a_{\alpha} $, $\alpha \in \Lambda$
			($\Lambda$ is an index set) implies $ c \leq \bigvee _{i=1}^{n}
			a_{\alpha_{i}} $, where $ n \in \mathbb{Z}^{+} $. A lattice $L$ is \textit{compactly generated} if every
			element of $L$ is a join of compact elements of $L$.
				\end{defn}
			\begin{defn} [ \cite{AAJ, A, SJ, WD}] A multiplicative lattice $L$ is \textit{$1$-compact} if $1$ is a compact element of $L$. A multiplicative lattice $L$ is a {\it $c$-lattice} if $L$ is 1-compact, compactly generated multiplicative lattice in which the product of two compact elements is compact. \end{defn}

		 \textbf{Throughout this paper, $S$ is a multiplicatively closed subset of a $V$-lattice $L$ such that $1 \in S$ and $0 \notin S$.}

		\begin{defn}[Sarode, Patil and Joshi \cite{SJP}]
		Let $S$ be a multiplicatively closed subset of a $V$-lattice $L$. A proper element $p$ of $L$ with $t \not\leqq p$ for all $t \in S$, is called an {\it $S$-prime} element of $L$ if there exists an element $s \in S$ such that for all $a, b  \in L$ with $a \cdot b \leq p$, we have $s \cdot a \leq p$ or $s \cdot b \leq p$.

		If we take $S=\{1\}$, then we get the classical definition of prime elements. That is, 	an element $p \neq 1$ of a $V$-lattice $L$  is a \textit{prime} element, if $a\cdot b \leq p$ implies $a \leq p$ or $b \leq p$ for all $a, ~b \in L$.

		 Equivalently, in a $c$-lattice $L$, an element $p \neq 1$  is a \textit{prime} element, if $a\cdot b \leq p$ implies $a \leq p$ or $b \leq p$ for  $a, ~b \in L_{\star}$, the set of all compact elements.

		The set of all $S$-prime elements of $L$ is denoted by $Spec_s(L)$. If $S = \{1\}$, then $Spec_s(L)$ is denoted by $Spec(L)$, the set of all prime elements of $L$.
	\end{defn}

	The following result is due to Hamed and Malek \cite{HM}.

	\begin{thm}[Hamed and Malek \cite{HM}] \label{1.20.}
	Let $R$ be a commutative ring with identity, $S $ be a multiplicative subset of $R$ and $P$ be an ideal of $R$ disjoint from $S$. Then $P$ is $S$-prime if and only if there exists $s \in S$, such that for all ideals $I, J$ of $R$, if $IJ \subseteq P$, then $sI \subseteq P$ or $sJ \subseteq P$.
	\end{thm}

	\begin{thm}[Sarode, Patil and Joshi \cite{SJP}]
	Let $R$ be a commutative ring with identity and $S$ a multiplicative subset of $R$. Let $L=Id(R)$ denote the multiplicative lattice of all ideals of $R$, and define
	$S_L=\{(a)\in L \mid a\in S\}.$ Then an ideal $P$ of $R$ is an $S$-prime ideal of $R$ if and only if $P$ (viewed as an element of $L$) is an $S_L$-prime element of $L$.
	\end{thm}

	Thus the lattice-theoretic notion of $S$-prime elements in $L$  extends the ring-theoretic notion of $S$-prime ideals in $R$.

		\begin{defn} [ \cite{D, SJ, WD}]
			Let $a,b $ be any elements of    a multiplicative lattice $L$. Then $(a:b) =
			\bigvee \{x \;|\;  x\cdot b \leq a\}$. Note that $x \cdot b \leq a
			\Leftrightarrow x \leq (a:b)$.  Let $L$ be a $c$-lattice and $a\in L$. The \textit{radical} of $a$ is denoted by $\sqrt a$ is defined as $\sqrt a =
			\bigvee \{x\in L_* \;|\;  x^n \leq a ~ \textrm{for some }~ n\in \mathbb{Z}^+\}$. \end{defn}
		Using the fact that the compact elements of $Id(R)$ are precisely the finitely generated ideals of $R$, we have the following.

		\begin{lem}\label{lem:radical-equals-ring}
			Let $R$ be a commutative ring with $1$, and let $L=Id(R)$ be the multiplicative lattice of ideals of $R$. For $A\in L$,  the lattice-theoretic radical
			$	\sqrt{A}	=\bigvee\{\,X\in L_* \mid X^n \leq A \text{ for some } n\in\mathbb{Z}^+ \}$ coincides with $\sqrt{A} = \{\, r\in R \mid r^n \in A \text{ for some } n\in\mathbb{Z}^+ \}$,  i.e., the lattice-theoretic radical coincides with the usual radical ideal of $A$ in $R$.
		\end{lem}

\begin{proof}
If $X$ is a finitely generated ideal with $X^n\subseteq A$, then $x^n\in A$ for every $x\in X$, so $X$ is contained in the radical ideal of $A$; hence the lattice-theoretic radical is contained in the radical ideal. Conversely, if $r^n\in A$, then $(r)$ is compact and $(r)^n=(r^n)\subseteq A$, so $r$ lies in the lattice-theoretic radical.
\end{proof}

\section{S-Noetherian Lattices}

Several authors, including Ward and Dilworth \cite{WD}, have studied \textit{multiplicative lattices} in connection with the ideal theory of rings. To generalize the study  of Noetherian rings to a broader order-theoretic framework, Jayaram\cite{JG} studied the concept of a \textit{Noetherian  lattice}. It is well known that the ideal lattice $Id(R)$ of a commutative ring $R$ forms a compactly generated, modular lattice in which the greatest element $1$ is compact. In fact, the notion of a Noetherian lattice is a true generalization of that of a Noetherian ring, as evidenced by the equivalence
$Id(R) $ is a Noetherian lattice  if and only if $R$  is a Noetherian ring.
This correspondence establishes a strong structural bridge between ring-theoretic and lattice-theoretic finiteness conditions.

In this section, we introduce an \textit{$S$-version of Noetherian lattices}, extending the classical framework by incorporating a multiplicatively closed subset $S$ of a $V$-lattice $L$. Motivated by the notion of  \linebreak $S$-Noetherian rings in commutative algebra, we define an \textit{$S$-Noetherian multiplicative lattice} and an analogue of the ascending chain condition (ACC) in terms of the \textit{$S$-stationary property}, thereby providing a flexible generalization  within the lattice context.

The study of $S$-Noetherian  lattices therefore lies at the intersection of \textit{order theory} and the \textit{ideal theory of commutative rings}, serving as a unifying framework to generalize classical algebraic concepts. This perspective offers an elegant order-theoretic setting to investigate extensions of Noetherian properties, finiteness conditions, and chain conditions beyond the algebraic confines of ring theory.

\begin{defn}[D. D. Anderson \cite{A}]
	Let $a, b, m$ be elements of a multiplicative lattice $L$.
	\begin{enumerate}
		\item An element $m$ is said to be\textit{ meet principal} if $a \wedge mb = m((a:m) \wedge b)$ for all $a,b \in L$.
		\item An element $m$ is said to be \textit{join principal} if $a \vee (b :m) = (am \vee b): m$ for all $a,b \in L$.
		\item An element $m$ is said to be  \textit{principal} if $m$ is both meet principal and join principal.
	\end{enumerate}

\end{defn}

\begin{defn} [D. D. Anderson \cite{A}]
	A multiplicative lattice $L$ is said to be an \textit{$r$-lattice}, if  $L$ is a modular, principally generated, and compactly generated lattice with $1$ as a compact element of $L$.
\end{defn}

The following definition of $S$-Noetherian ring is due to Anderson and Dumitrescu \cite{AT}.

\begin{defn} [D. D. Anderson and Tiberiu Dumitrescu \cite{AT}]
	Let $S$ be a multiplicative subset of a commutative ring $R$ with 1. An ideal $I$ of $R$ is said to be \textit{$S$-finite}, if there exists some finitely generated ideal $J$ of $R$ and some $s \in S$ such that $sI \subseteq J \subseteq I$. A commutative ring $R$ with 1 is said to be \textit{$S$-Noetherian} if each ideal of $R$ is $S$-finite.
\end{defn}

Now, analogously, we define  $S$-Noetherian lattice as follows.

\begin{defn}
	Let $S$ be a multiplicative subset of a multiplicative lattice $L$. An element $a$ of $L$ is said to be \textit{$S$-compact}, if there exists some compact element $b$ in $L$ and some $s \in S$ such that $s\cdot a \leq b \leq a$. An $r$-lattice $L$ is said to be an \textit{$S$-Noetherian lattice} if each element of $L$ is $S$-compact.
\end{defn}

In the following result, we connect the $S$-Noetherian rings with $S_L$-Noetherian lattices.

\begin{thm}\label{thm:SL-Noeth}
	Let $R$ be a commutative ring with $1$ and let $S\subseteq R$ be a multiplicatively closed set.
	Set $L:=Id(R)$ and $S_L:=\{\, (s)\in L \mid s\in S \,\}$.
	Then the following are equivalent:
	\begin{enumerate}
		\item $R$ is $S$-Noetherian (i.e., every ideal $I$ is $S$-finite).
		\item $L=Id(R)$ is $S_L$-Noetherian as a multiplicative lattice (i.e., every element $I\in L$ is $S_L$-compact).
	\end{enumerate}
\end{thm}

\begin{proof}
	Note that $Id(R)$ is an $r$-lattice  (every principal ideal of $R$ is a principal element of $Id(R)$; see \cite{A}). Since $1\in S$ and $0\notin S$, we have $(1)\in S_L$ and $(0)\notin S_L$; moreover $S_L\subseteq L_*$ is multiplicatively closed, as $(s)(t)=(st)$, and $(s)\cdot I=sI$ for every ideal $I$ of $R$.

	$(1)\Rightarrow(2)$:
	Let $I\in L$. As $R$ is $S$-Noetherian, there exist $s\in S$ and a finitely generated ideal $J\subseteq I$ with $sI\subseteq J$.
	In $L$, $(s)\cdot I  \leq J \leq I$, and $J$ is compact. Hence $I$ is $S_L$-compact. As $I$ was arbitrary, $L$ is $S_L$-Noetherian.

	$(2)\Rightarrow(1)$:
	Let $I$ be an ideal of $R$. Since $L$ is  $S_L$-Noetherian, there exist $(s)\in S_L$ and a compact $b\in L$ with $(s)\cdot I\leq b\leq I$.
	It is well known that the compact elements in $Id(R)$ are precisely the  finitely generated ideals of $R$, we get $sI\subseteq b\subseteq I$ with $b$ as a finitely generated ideal. Thus $I$ is $S$-finite. Since $I$ was arbitrary, $R$ is $S$-Noetherian.
\end{proof}
\begin{rem}\label{remark}
	Let $L$ be $r$-lattice. Then every principal element of $L$ is compact, and moreover the product of two compact elements is compact in an $r$-lattice  $L$. For more details; see \cite{A}. In particular, every $r$-lattice is a $c$-lattice.
\end{rem}

\begin{lem} \label{5.51.}
	Let $S$ be a multiplicatively closed subset of an $r$-lattice $L$. Then the following are true.
	\begin{enumerate}
		\item If $j$ is a principal element and $i$ is an $S$-compact element of $L$, then $i\cdot j$ is $S$-compact.
		\item If $k$ is a compact element and $i$ is an $S$-compact element of $L$, then $k \vee i$ is $S$-compact.
	\end{enumerate}
\end{lem}

\begin{proof}
	\begin{enumerate}
		\item Since $i$ is an $S$-compact element of $L$, there exist a compact element $b$ and some $s \in S$ such that $s \cdot i \leq b \leq i$. Therefore we have, $s\cdot i \cdot j \leq b\cdot j \leq i\cdot j$. As $j$ is a principal element, by Remark \ref{remark}, $j$ is a compact element of $L$. Moreover, since $b$ is compact, $b\cdot j$ is a compact element of $L$.  Therefore 	$i\cdot j$ is an $S$-compact element of $L$.

		\item Since $i$ is an $S$-compact element of $L$, there exist a compact element $b$ and some $s \in S$ such that $s \cdot i \leq b \leq i$. Therefore we have, $(k \vee s i) \leq ( k \vee b) \leq (k \vee i)$. Since $s\cdot k\leq k$, this gives $s\cdot (k \vee i) \leq  (k \vee s \cdot i) \leq (k \vee b) \leq (k \vee i)$.  Thus $k \vee i$ is an $S$-compact element of $L$.\end{enumerate}\end{proof}

\begin{definition}[Jayaram\cite{JG}]
	A multiplicative lattice $L$ is a Noetherian  lattice, if $L$ is an	$r$-lattice with ACC. Equivalently, $L$ is Noetherian if and only if every element is compact.
\end{definition}

As established earlier, the ideal lattice $Id(R)$ of a Noetherian (or $S$-Noetherian) ring $R$ possesses the Noetherian (or $S$-Noetherian) property. This correlation naturally invites an investigation into the specific distinction between $S$-Noetherian and classical Noetherian lattices. In this context, a foundational example provided in \cite[Example 2.5]{SAK} serves the purpose. Specifically, consider the following example.

\begin{exa}[\cite{SAK}]Let $L = \mathrm{Id}(R)$ denote the ideal lattice of the polynomial ring $R = F[x_1, x_2, \dots, x_n, \dots]$ over a field $F$. It is well-known that the lattice $L$ is not Noetherian, as it admits the strictly ascending infinite chain of ideals $(x_1) \subsetneqq (x_1, x_2) \subsetneqq (x_1, x_2, x_3) \subsetneqq \dots $ However, $R$ is an $S$-Noetherian ring for the multiplicative set $S = R \setminus \{0\}$. Consequently, by applying Theorem \ref{thm:SL-Noeth}, it follows that the lattice $L$ is $S_L$-Noetherian, where $S_L = \{ (s) \mid s \in S \}$.\end{exa}

Hamed and Hizem \cite{HH} have defined $S$-stationary for sequences of submodules of a module. Analogously, we define $S$-stationary for sequences of elements of a $V$-lattice $L$.

\begin{defn}
	Let $S$ be a multiplicatively closed subset of a $V$-lattice $L$. An ascending chain of elements $i_1 \leq i_2 \leq i_3 \leq \cdots \leq i_m  \leq \cdots $ is called \textit{$S$-stationary} if there exist $s \in S$ and $ n \in \mathbb{Z}^{+}$ such that $s  i_m \leq i_n$ for all $m \geq n$.

    If $S$ is a multiplicatively closed subset of a $V$-lattice $L$, then $L$ is said to satisfy the \textit{$S$-stationary property}, if every ascending chain of elements is $S$-stationary.

	Let $S$ be a multiplicatively closed subset of a $c$-lattice $L$ and let $T\subseteq L$. An element $m$ of $T$ is called an \textit{$S$-maximal element of $T$} if there exists an element $s\in S$ such that for each $a\in T$, if $m\leq a$, then $sa\leq m$.
\end{defn}

\begin{definition}\label{def:sat}
Let $L$ be a $c$-lattice. For $a\in L$, the \textit{$S$-saturation} of $a$ is
$a_S=\bigvee\{x\in L\mid sx\leq a \text{ for some } s\in S\}$. An element $a$ is called \textit{$S$-saturated} if $a=a_S$.
\end{definition}
Note that if $S=\{1\}$, then every element of $L$ is $S$-saturated. 
\begin{lem}\label{lem:satprops}
Let $L$ be a $c$-lattice and $a,b,y\in L$.
\begin{enumerate}
\item $a\leq a_S$; if $a\leq b$, then $a_S\leq b_S$; and if $sx\leq a$ for some $s\in S$, then $x\leq a_S$.
\item For $x\in L_*$: $x\leq a_S$ if and only if $sx\leq a$ for some $s\in S$.
\item $(a_S)_S=a_S$.
\item $(a\wedge b)_S=a_S\wedge b_S$.
\item If $a$ is $S$-saturated, then so is $(a:y)$.
\item $a_S=1$ if and only if $t\leq a$ for some $t\in S$.
\end{enumerate}
\end{lem}
\begin{proof}
(1) is clear (take $s=1$). 

(2) Clearly, the set $D_a=\{x\in L\mid sx\leq a$ for some $s\in S\}$ is closed under finite joins. For this,  if $x_1, x_2 \in D_a$, then  $s_1x_1\leq a$ and $s_2x_2\leq a$ for some $s_1, s_2 \in S$.  Clearly,  $s_1s_2(x_1\vee x_2)\leq s_1x_1 \vee s_2x_2\leq  s_1a\vee s_2a\leq a$. Hence a compact $x\leq a_S=\bigvee D_a$ lies below a single member $x'$ of $D_a$, and $sx\leq sx'\leq a$. 

(3) If $x\in L_*$ and $x\leq (a_S)_S$, then $sx\leq a_S$ for some $s\in S$ by (2); as $sx$ is compact, $s'sx\leq a$ for some $s'\in S$, so $x\leq a_S$. Since $L$ is compactly generated, $(a_S)_S\leq a_S$. 

(4)  $(a\wedge b)_S \leq a_S \wedge b_S$ follows from (1). If $x\in L_*$ and $x\leq a_S\wedge b_S$, then $s_1x\leq a$ and $s_2x\leq b$, so $s_1s_2x\leq a\wedge b$ and $x\leq (a\wedge b)_S$.  Since $L$ is compactly generated, $a_S\wedge b_S\leq (a\wedge b)_S$.

(5) If $x\in L_*$ and $x\leq (a:y)_S$, then $sxy\leq a$ for some $s\in S$, so $xy\leq a_S=a$ by (1), i.e., $x\leq (a:y)$. Since $L$ is compactly generated, $(a:y)_S\leq (a:y)$, and $(a:y)\leq (a:y)_S$ by (1).

(6) Since $1$ is compact, $1\leq a_S$ if and only if $s\cdot 1\leq a$ for some $s\in S$, by (2).
\end{proof}

\begin{lem}\label{lem:A}
Let $L$ be an $S$-Noetherian lattice. Then for every $a\in L$, there exists $v\in S$ such that $v\, a_S\leq a$.
\end{lem}
\begin{proof}
Since $L$ is $S$-Noetherian, $a_S$ is $S$-compact.  Hence, there are $u\in S$ and a compact $c$ with $u\,a_S\leq c\leq a_S$. By Lemma \ref{lem:satprops}(2), $s_c c\leq a$ for some $s_c\in S$. Put $v=us_c\in S$. Then $v\,a_S=s_c(u\,a_S)\leq s_c c\leq a$.
\end{proof}

\begin{lem}\label{lem:B}
Let $L$ be a $c$-lattice satisfying the $S$-stationary property. Then every ascending chain of $S$-saturated elements of $L$ is stationary. Consequently, every nonempty set of $S$-saturated elements of $L$ has a maximal element.
\end{lem}
\begin{proof}
Let $a_1\leq a_2\leq\cdots$ be $S$-saturated. There are $s\in S$ and $n$ with $sa_m\leq a_n$ for all $m\geq n$; by Lemma \ref{lem:satprops}(1), $a_m\leq (a_n)_S=a_n$, so $a_m=a_n$ for all $m\geq n$. If a nonempty set of $S$-saturated elements had no maximal element, we could choose a strictly ascending chain in it, a contradiction.
\end{proof}

\begin{thm} \label{5.11a.}
	Let $L$ be an $r$-lattice. Consider the following statements:

	\begin{enumerate}
		\item $L$ is  an $S$-Noetherian lattice.

		\item Every element of $L$ is $S$-compact.

		\item $L$  satisfies $S$-stationary property.
		\item\label{S-MC} Every non-empty subset of $L$ has $S$-maximal element.
	\end{enumerate}
Then $(1)$, $(2)$ and $(4)$ are equivalent, and each of them implies $(3)$. If $S$ is countable, then all four statements are equivalent.
\end{thm}

\begin{proof}
	$(1) \Leftrightarrow  (2)$ follows from the definition of an $S$-Noetherian lattice.

	$(2) \Rightarrow (3)$ Suppose that every element of $L$ is $S$-compact. We will show that $L$ satisfies the \linebreak $S$-stationary property. Let $b_1 \leq b_2 \leq b_3 \leq \cdots$ be any ascending chain of elements in $L$. Let $j=\bigvee_{n\geq 1} b_n$. Since $j$ is $S$-compact, we have $s \cdot j \leq k \leq j$ for some compact element $k$ and some $s \in S$. Therefore $s \cdot (b_1 \vee b_2 \vee b_3 \vee \cdots) \leq k \leq (b_1 \vee b_2 \vee b_3 \vee \cdots)$. Since $k$ is compact and $k \leq (b_1 \vee b_2 \vee b_3 \vee \cdots)$, we have $k \leq b_1 \vee b_2 \vee b_3 \vee \cdots \vee b_n =b_n$, as $b_1 \leq b_2 \leq b_3 \leq \cdots \leq b_n$. So $s \cdot (b_1 \vee b_2 \vee b_3 \vee \cdots) \leq b_n \leq (b_1 \vee b_2 \vee b_3 \vee \cdots)$. This gives $s \cdot b_1 \vee s \cdot b_2 \vee s \cdot b_3 \vee \cdots \leq  b_n$. Thus,  $s b_m \leq b_n$ for all $m \geq n$. Hence $L$  satisfies $S$-stationary property.

$(2)\Rightarrow(4)$: Let $T$ be a non-empty subset of $L$. 
Since $T\neq\emptyset$, the set $\{x_S\mid x\in T\}$ is nonempty, and by Lemma \ref{lem:satprops}(3) it consists of $S$-saturated elements. By $(2)\Rightarrow(3)$ and Lemma \ref{lem:B}, it has a maximal element $m_S$, with $m\in T$.

Let $a\in T$ with $m\leq a$. Then $m_S\leq a_S$, hence $a_S=m_S$ by maximality, and so $a\leq a_S=m_S$. By Lemma \ref{lem:A}, there is $v\in S$ with $v\,m_S\leq m$. Therefore $va\leq v\,m_S\leq m$ for every $a\in T$ with $m\leq a$, i.e., $m$ is an $S$-maximal element of $T$.

$(3)\Rightarrow(4)$ when $S$ is countable: Since $S$ is countable, we write $S=\{s_1,s_2,s_3,\dots\}$. Put $u_n=s_1s_2\cdots s_n\in S$. Then $u_1\geq u_2\geq\cdots$, and
$$
(\dagger)\qquad \text{for each } j\geq 1 \text{ and each } k\geq j,\quad u_k\leq s_j,
$$
since $u_k=s_j\cdot\prod_{i\leq k,\, i\neq j}s_i\leq s_j$ (recall $ab\leq a$ in $L$). This is the only place where the countability of $S$ is used.

Suppose that a non-empty subset $T$ of $L$ has no $S$-maximal element. Choose $x_1\in T$. Having chosen $x_n\in T$, since $x_n$ is not $S$-maximal in $T$, applying this to the element $u_n\in S$, there is $x_{n+1}\in T$ with $x_n\leq x_{n+1}$ and $u_nx_{n+1}\nleq x_n$. This gives an ascending chain $x_1\leq x_2\leq\cdots$ in $T$. By $(3)$, there are $s\in S$ and $n\in\mathbb{Z}^+$ such that $sx_m\leq x_n$ for all $m\geq n$. Write $s=s_j$ and take $k\geq\max\{j,n\}$. By $(\dagger)$, $u_k\leq s_j$, and hence $
u_kx_{k+1}\leq s_jx_{k+1}\leq x_n\leq x_k$,
contradicting the choice of $x_{k+1}$. Therefore $T$ has an $S$-maximal element.

	$(4) \Rightarrow  (2)$: Assume that every non-empty subset of $L$ has an $S$-maximal element. Let $x\in L$. If $x=0$, then $x$ is compact and in this case there is nothing to prove; so assume that $x\neq 0$.

	Since $L$ is an $r$-lattice, it is compactly generated and hence there exists a family $\{a_\alpha\}_{\alpha\in\Lambda}\subseteq L_*$
	of compact elements such that
	$x=\bigvee_{\alpha\in\Lambda} a_\alpha$. Let
	$T=\left\{\bigvee_{i=1}^n a_{\alpha_i}\mid
	n\in\mathbb{N},\ \alpha_1,\dots,\alpha_n\in\Lambda\right\}$.
	Then $T$ is a non-empty subset of $L$, and every element of $T$ is compact,
	since finite join of compact elements is compact.
	Moreover, $\bigvee T=x$.

	By the assumption, $T$ has an $S$-maximal element; say, $k\in T$ be such an element.
	Then $k$ is compact and $k\leq x$. Since $k$ is $S$-maximal in $T$, there exists $s\in S$ such that $sa\leq k$ for every $a\in T$ with $k\leq a$. For any $a\in T$, we have $a\vee k\in T$ and $k\leq a\vee k$, so $sa\leq s(a\vee k)\leq k$. As $x=\bigvee T$ and multiplication distributes over joins, $sx=\bigvee_{a\in T}sa\leq k$.
	 Therefore, $sx\leq k\leq x$, with $k$ compact.
	Hence $x$ is an $S$-compact element of $L$. Since $x\in L$ was arbitrary, every element of $L$ is $S$-compact.
	\end{proof}

In the sequel, we obtain the $S$-version of Cohen's result. For details of Cohen's result; see Kaplansky \cite[Theorem 8]{kap}.

\begin{rem}\label{princ-prime}
	\begin{enumerate}
	\item Observe that if $L$ is an $r$-lattice, then to establish that an element
	$p \in L$ is prime, it is enough to verify that
	$x \cdot y \leq p \;\implies\; x \leq p \ \text{or}\ y \leq p
	$
	holds whenever $x, y$ are principal elements of $L$.

	\item Note that if $a$ is not an $S$-compact element of $L$, then $t\nleq a$ for all $t\in S$. For, if $t_1\leq a$ for some $t_1\in S$, then $t_1a\leq t_1\leq a$, where $t_1\in S\subseteq L_{\star}$, the set of compact elements. Therefore, $a$ is  $S$-compact, a contradiction.
	\end{enumerate}
\end{rem}

\begin{thm}[Cohen--Kaplansky $S$-version]\label{thm:CK}
	Let $S$ be a multiplicatively closed subset of an $r$-lattice $L$. Then $L$ is an $S$-Noetherian lattice if and only if  every $S$-prime element of $L$  is  $S$-compact, if and only if every prime element of $L$ is $S$-compact.
\end{thm}

\begin{proof}If $L$ is an $S$-Noetherian lattice, then every element is $S$-compact and hence every $S$-prime element and every prime element is $S$-compact.
	
	Conversely, assume that every $S$-prime element of $L$ is $S$-compact. We first show that every prime element $p$ of $L$ is $S$-compact. There are two cases.
	
	\emph{Case (i): $t\leq p$ for some $t\in S$.} Then $t\cdot p\leq t\wedge p=t\leq p$, and $t$ is compact, since $S\subseteq L_*$. Hence $t\cdot p\leq t\leq p$ with $t\in S$ and $t$ compact, i.e., $p$ is $S$-compact (cf. Remark \ref{princ-prime}(2)).
	
	\emph{Case (ii): $t\nleq p$ for all $t\in S$.} Then $p$ is an $S$-prime element of $L$ with $s=1$. For this,  $p\neq 1$, $t\nleq p$ for all $t\in S$, and if $a\cdot b\leq p$, then $a\leq p$ or $b\leq p$, i.e., $1\cdot a\leq p$ or $1\cdot b\leq p$. By the hypothesis, $p$ is $S$-compact.
	
	Thus every prime element of $L$ is $S$-compact. Consequently, it suffices to prove that if every prime element of $L$ is $S$-compact, then $L$ is $S$-Noetherian.

	 Let $T = \{a' \in L\; |\; a'$ is not  $S$-compact$\}$. If $T = \emptyset$, then all elements of $L$ are $S$-compact and by Theorem \ref{5.11a.}, $L$ is an $S$-Noetherian lattice.

Now, assume that $T\neq\emptyset$. Let $C\subseteq T$ be a nonempty chain and $a=\bigvee C$. Suppose that $a$ is $S$-compact, say $sa\leq b\leq a$ with $s\in S$ and $b$ compact. Since $b$ is compact and $C$ is a chain, $b\leq c_0$ for some $c_0\in C$. Then $sc_0\leq sa\leq b\leq c_0$, so $c_0$ is $S$-compact, a contradiction. Hence $a\in T$ is an upper bound of $C$. By Zorn's Lemma, $T$ contains a maximal element, say $p$. Then $p$ is not $S$-compact; in particular $p\neq 1$ (as $1$ is compact) and $t\nleq p$ for all $t\in S$ by Remark \ref{princ-prime}(2). If $p$ is prime, then $p$ is $S$-compact by the hypothesis, contradicting $p\in T$.

	If $p$ is not a prime element, then, by Remark \ref{princ-prime}, there exist principal elements $x, y$ in $L$ such that $xy \leq p$ with $x \not\leqq p$ and $y \not\leqq p$. Let $A = x \vee p$ and $B = (p : x)$. Clearly, $p<A,B$, as $x\leq A$, but $x\nleq p$ and $y\leq B$, but $y\nleq p$. Therefore, $A,B\notin T$. Hence, $A$ and $B$ are $S$-compact elements of $L$. Since $x$ is principal, $p \wedge x = x \cdot (p: x) = x \cdot  B$. By Lemma \ref{5.51.}, $x \cdot B$ is $S$-compact and hence $p \wedge x$ is $S$-compact. 
	Since $A$ is $S$-compact, there are $s\in S$ and a compact $c$ with $sA\leq c\leq A=p\vee x$. As $c$ is compact and $p$ is a join of compact elements, $c\leq p_0\vee x$ for some compact $p_0\leq p$. Since $L$ is modular and $p_0\leq p$, we get $sp\leq p\wedge sA\leq p\wedge(p_0\vee x)=p_0\vee(p\wedge x)$. Since $p\wedge x$ is $S$-compact, there are $s'\in S$ and a compact $d$ with $s'(p\wedge x)\leq d\leq p\wedge x$. Hence $(s's)p=s'(sp)\leq s'(p_0\vee (p\wedge x))=s'p_0 \vee s'(p\wedge x)\leq p_0\vee d\leq p$, where $p_0\vee d$ is compact. Thus $p$ is $S$-compact, contradicting $p\in T$.
	Thus, $T=\varnothing$, i.e., every element of $L$ is $S$-compact, and by Theorem~\ref{5.11a.}, $L$ is $S$-Noetherian.
\end{proof}

\section{$S$-primary decomposition in $S$-Noetherian lattices}
In this section, we have given the existence and uniqueness of  $S$-primary decomposition in an $S$-Noetherian lattice.  This extends the results of $S$-Noetherian rings to $S$-Noetherian lattices.  Recall that every $S$-Noetherian lattice is an $r$-lattice, hence a $c$-lattice (Remark \ref{remark}).

\begin{defn}[\cite{M,V}]
	Let $R$ be a commutative ring with identity and $S$ a multiplicatively closed subset of $R$.
	A proper ideal $Q$ of $R$ is called an \emph{$S$-primary ideal} of $R$ if
	$S\cap Q=\varnothing$ and there exists $s\in S$ such that for all $a,b\in R$,
	$ab\in Q$ implies that $sa\in Q$ or $sb\in \sqrt{Q}$.
\end{defn}

Analogously, we define an $S$-primary element in a multiplicative lattice as follows.

\begin{definition}  Let $S$ be a multiplicatively closed subset of a $c$-lattice $L$. A proper element $q$ of $L$ such that $t\nleq q$ for all $t\in S$ is called an \textit{$S$-primary} element if there exists an element $s\in S$ such that for all $c,d \in L$, $cd \leq q$  implies $sc \leq q$ or $sd \leq \sqrt q$.

 If $q$ is an $S$-primary element of a $c$-lattice $L$, then $\sqrt{q}=p$ is an $S$-prime element (Lemma \ref{lem:radSprime})and $q$ is called an $S$-$p$-primary element of $L$.
\end{definition}

\begin{lem}\label{lem:radSprime}
Let $L$ be a $c$-lattice and $q$ an $S$-primary element of $L$. Then $\sqrt q$ is an $S$-prime element of $L$ (with the same element $s\in S$ as for $q$).
\end{lem}
\begin{proof}
If $t\leq\sqrt q$ for some $t\in S$, then, $t$ being compact, $t^n\leq q$ with $t^n\in S$, a contradiction. Let $s$ be as in the definition for $q$ and let $a,b\in L$ with $ab\leq\sqrt q$ and $sa\nleq\sqrt q$. Since $sa=\bigvee\{sx\mid x\leq a,\ x\in L_*\}$, there is a compact $x\leq a$ with $sx\nleq\sqrt q$. Let $y\leq b$ be compact. Then $xy$ is compact and $xy\leq\sqrt q$, so $x^ny^n\leq q$ for some $n$. If $sx^n\leq q$, then $(sx)^n\leq sx^n\leq q$, and $sx\leq\sqrt q$, a contradiction. Hence $sy^n\leq\sqrt q$, so $(sy^n)^m\leq q$ for some $m$, and $(sy)^{nm}\leq (sy^n)^m\leq q$, i.e., $sy\leq\sqrt q$. Therefore $sb=\bigvee sy\leq\sqrt q$.
\end{proof}

\begin{rem}
	\begin{enumerate}

		\item If $L$ is a $c$-lattice and $S=\{1\}$, then the primary elements of $L$ and the $S$-primary elements of $L$ coincide.
		\item  If $q$ is a primary element of a $c$-lattice $L$ and $S$ is a multiplicatively closed subset of $L$ such that $t\nleq q$ for all $t\in S$, then $q$ is an $S$-primary element of $L$.

	\end{enumerate}
\end{rem}

\begin{exa}
	\begin{enumerate}
		\item Consider the lattice $L=Id(\mathbb{Z}_{12})$, the ideal lattice of the ring $\mathbb{Z}_{12} $.  Observe that  $L$ is a c-lattice with the multiplication as the multiplication  of ideals. Let $S = \{(1), ~(3) \}$. Then
		\begin{enumerate}
			\item  $(4)$ is an $S$-primary element of $L$ but not an $S$-prime element of $L$.

			\item $(6)$ is an $S$-primary element of $L$ but not a primary element of $L$.
		\end{enumerate}

		\item Let $L=Id(\mathbb{Z}_{12})$. Let $S = \{(1), ~(4) \}$. Then $(4)$ is a primary element of $L$ but not $S$-primary (in particular, not $S$-prime).

	\end{enumerate}
\end{exa}

The following result gives the characterization of $S$-primary ideals.
\begin{thm}[\cite{M}, Theorem 2.11]\label{2.11.}
	Let $R$ be a commutative ring with identity, $S$ a multiplicative subset of $R$ and $P$ an ideal of $R$ disjoint from $S$. Then $P$ is $S$-primary if and only if there exists $s \in S$, such that for all ideals $I, J$ of $R$, if $IJ \subseteq P$, then $sI \subseteq P$ or $sJ \subseteq \sqrt{P}$.
\end{thm}

\begin{thm}\label{thm:S-primary-lattice-converse}
	Let $R$ be a commutative ring with identity and $S$ a multiplicatively closed subset of $R$.
	Let $L=Id(R)$ denote the multiplicative lattice of all ideals of $R$, and define
	$S_L=\{(a)\in L \mid a\in S\}$.
	Then an ideal $Q$ of $R$ is an $S$-primary ideal of $R$ if and only if $Q$ (viewed as an element of $L$)
	is an $S_L$-primary element of $L$.
\end{thm}

\begin{proof}
	Suppose that $Q$ is an $S$-primary ideal of $R$.
	Then $S\cap Q=\varnothing$. This shows that  $t\not\leqq Q$ for all $t\in S_L$.
	Further, by Theorem \ref{2.11.}, there exists $s\in S$ such that for all $I,J\in L$, $IJ\leq Q$ implies
	$(s)I\leq Q$ or $(s)J\leq \sqrt{Q}$, where the radical of $Q$ is the same in $L$ and in $R$ by Lemma \ref{lem:radical-equals-ring}.
	Hence $Q$ is an $S_L$-primary element of $L$.

	Conversely, assume that $Q$ is an $S_L$-primary element of $L$.
	By the definition, $t\not\leqq Q$ for all $t\in S_L$.
	If there exists $a\in S\cap Q$, then $(a)\in S_L$ and $(a)\leq Q$, a contradiction.
	Thus $S\cap Q=\varnothing$.

	Now let $I,J$ be ideals of $R$ such that $IJ\subseteq Q$.
	Equivalently, $IJ\leq Q$ in $L$.
	Since $Q$ is $S_L$-primary, there exists $(s)\in S_L$, for some $s\in S$, such that
	$(s)I\leq Q$ or $(s)J\leq \sqrt{Q}$.
	That is, $sI\subseteq Q$ or $sJ\subseteq \sqrt{Q}$, by Lemma \ref{lem:radical-equals-ring}.
	Hence $Q$ is an $S$-primary ideal of $R$.
\end{proof}

\begin{definition}
	Let $L$ be a $c$-lattice and $S \subseteq L$ be a multiplicatively closed subset. Let $a \in L$ be an element such that its $S$-saturation (Definition \ref{def:sat}) $a_S < 1$. We say that $a$ admits an \textit{$S$-primary decomposition } if it can be represented as a finite meet of $S$-primary elements
	$a = \bigwedge_{i=1}^{n} q_i$ where each $q_i$ is an $S$-$p_i$-primary element of $L$. Such a decomposition is said to be \textit{minimal }  if the following two conditions are satisfied:
	\begin{itemize}
		\item The $S$-saturations of the radicals are distinct for all $i, j \in \{1, 2, \dots, n\}$:
		$(p_i)_S \neq (p_j)_S \;\; \text{for } i \neq j$.
		\item   For each $i \in \{1, 2, \dots, n\}$:
		$\left( \bigwedge_{j \neq i} q_j \right)_S \not\leq (q_i)_S$ (Equivalently, for each $i \in \{1, \ldots, n\}$, we can find $s_i \in S$ and $x_i \in L_*$ such that
		$s_i x_i \leq \bigwedge_{j \in \{1, \ldots, n\} \setminus \{i\}} q_j$ but $s x_i \nleq q_i$ for all $s \in S$.)
	\end{itemize}
\end{definition}

\begin{prop}\label{prop:finite-intersection-S-primary}
	Finite meet of $S$-$p$-primary elements is an $S$-$p$-primary element in a $c$-lattice $L$, where $p$ is an $S$-prime element of $L$.
\end{prop}

\begin{proof}
	Let $q_1,q_2,\dots,q_n$ be $S$-$p$-primary elements of $L$.
	Then $t\nleq q_i$ for all $t\in S$ and for each $i=1,2,\dots,n$.
	Hence $t\nleq \bigwedge_{i=1}^{n} q_i$ for all $t\in S$. Since $L$ is a $c$-lattice, we have
	$\sqrt{\bigwedge_{i=1}^{n} q_i}
	=\bigwedge_{i=1}^{n}\sqrt{q_i}
	=p$.
	Let $s_i\in S$ be the element in the definition of the $S$-primary element $q_i$, and put $u=s_1s_2\cdots s_n\in S$. Let $x,y\in L$ with $xy\leq \bigwedge_{i=1}^{n} q_i$ and $uy\nleq \bigwedge_{i=1}^{n} q_i$. Then $uy\nleq q_k$ for some $k$, and hence $s_ky\nleq q_k$, as $uy\leq s_ky$. Since $xy\leq q_k$, we get $s_kx\leq \sqrt{q_k}=p$, and so $ux\leq s_kx\leq p=\sqrt{\bigwedge_{i=1}^{n} q_i}$. Thus, with the single element $u$, $\bigwedge_{i=1}^{n} q_i$ is an $S$-$p$-primary element of $L$.
\end{proof}

\begin{definition}\label{def:Smaxl}
Let $L$ be a $c$-lattice and let $L^S=\{a\in L\mid t\nleq a \text{ for all } t\in S\}$ be the set of \textit{$S$-proper} elements of $L$. An element $m\in L^S$ is called an \textit{$S$-maximal element of $L$} if $m$ is an $S$-maximal element of the subset $T=L^S$, that is, there exists $s\in S$ such that for each $a\in L^S$, if $m\leq a$, then $sa\leq m$.
\end{definition}

\begin{rem}\label{rem:Smaxl}
	
\begin{enumerate}
\item Taking $T=L$ itself would be useless, since $m\leq 1$, the condition would give $s=s\cdot 1\leq m$, which no element of $L^S$ satisfies. Hence $T=L^S$ is used, just as a maximal element of $L$ is maximal in $L\setminus\{1\}$ and not in $L$.
\item If $S=\{1\}$, then $L^S=L\setminus\{1\}$ and the condition reads: $m\leq a\neq 1$ implies $a\leq m$. Thus the $\{1\}$-maximal elements of $L$ are exactly the maximal elements of $L$.
\item If $m$ is an $S$-maximal element of $L$, then $m_S$ is a maximal element of the poset $L^S$. Indeed, $m_S\in L^S$ by Lemma \ref{lem:satprops}(6), and if $a\in L^S$ with $m_S\leq a$, then $m\leq a$, so $sa\leq m$ and $a\leq m_S$.
\end{enumerate}
\end{rem}

\begin{lem}\label{lem:C}
	Let $L$ be a $c$-lattice, let $Q$ be a primary element of $L$ with $t\nleq Q$ for all $t\in S$, and let $v\in S$ and $j\in L$ with $vQ\leq j\leq Q$. Then $j$ is an $S$-primary element of $L$ (with the element $v$).
\end{lem}
\begin{proof}
	Since $j\leq Q$, we have $t\nleq j$ for all $t\in S$. Let $c,d\in L$ with $cd\leq j$. Then $cd\leq Q$, so $c\leq Q$ or $d\leq\sqrt Q$. If $c\leq Q$, then $vc\leq vQ\leq j$. Let $d\leq\sqrt Q$.  Assume that $x\leq d$ is compact.  Then $x^n\leq Q$ for some $n$, so $(vx)^n\leq vx^n\leq vQ\leq j$, and $vx\leq\sqrt j$, as $vx$ is compact. Hence $vd=\bigvee vx\leq\sqrt j$.
\end{proof}

\begin{prop}\label{prop:Smaxl}
	Let $S$ be a multiplicatively closed subset of an $S$-Noetherian lattice $L$. Let $q\in L$ such that $t\nleq q$ for all $t\in S$. If $\sqrt{q}=m$ is an $S$-maximal element of $L$, then $q$ is  an $S$-primary element of $L$.
\end{prop}
\begin{proof}
We prove the result using the following four claims. Put $M=m_S$. 

\emph{Claim 1: $\sqrt{q_S}=M$.} Let $x$ be a compact element of $L$. If $x\leq\sqrt{q_S}$, then $x^n\leq q_S$ for some $n$, so $sx^n\leq q$ for some $s\in S$ by Lemma \ref{lem:satprops}(2); hence $(sx)^n\leq sx^n\leq q$, so $sx\leq\sqrt q=m$ and $x\leq m_S=M$. 

Conversely, if $x\leq M$, then $sx\leq m=\sqrt q$ for some $s\in S$; as $sx$ is compact, $s^nx^n=(sx)^n\leq q\leq q_S$ for some $n$, so $x^n\leq (q_S)_S=q_S$ by Lemma \ref{lem:satprops}(1),(3), and $x\leq\sqrt{q_S}$. As $L$ is compactly generated, $\sqrt{q_S}=M$.

\emph{Claim 2: $M^N\leq q_S$ for some $N\geq 1$.} Since $L$ is $S$-Noetherian, there are $w\in S$ and a compact $c$ with $wM\leq c\leq M=\sqrt{q_S}$.  Hence $c^N\leq q_S$ for some $N$. Hence $w^NM^N=(wM)^N\leq c^N\leq q_S$, and $M^N\leq (q_S)_S=q_S$.

\emph{Claim 3: $q_S$ is a primary element of $L$.} By Lemma \ref{lem:satprops}(6), $q_S\neq 1$. Let $c,d\in L$ with $cd\leq q_S$ and $d\nleq\sqrt{q_S}=M$. Then $M<M\vee d$, so by Remark \ref{rem:Smaxl}(3), $M\vee d\notin L^S$, i.e., $t\leq M\vee d$ for some $t\in S$. Hence $t^N\leq (M\vee d)^N\leq M^N\vee d$, and $
t^Nc\leq M^Nc\vee dc\leq q_S$.
By Lemma \ref{lem:satprops}(1), $c\leq (q_S)_S=q_S$. Thus $cd\leq q_S$ implies $c\leq q_S$ or $d\leq\sqrt{q_S}$, i.e., $q_S$ is primary.

\emph{Claim 4: $q$ is $S$-primary.} By Lemma \ref{lem:A}, there is $v\in S$ with $vq_S\leq q\leq q_S$. Moreover,  $t\nleq q_S$ for all $t\in S$, since $q_S\neq 1$ (Lemma \ref{lem:satprops}(6)). By Claim 3 and Lemma \ref{lem:C} (applied with $Q=q_S$ and $j=q$), $q$ is an $S$-primary element of $L$.
\end{proof}

Now, we relate the known  concept of $S$-irreducible ideals in a commutative ring with the $S$-irreducible elements in a multiplicative lattice.

\begin{definition} (Definition 2.2 \cite{SAK})
	Let $R$ be a commutative ring with identity and let $S$ be a multiplicatively
	closed subset of $R$.
	A proper ideal $Q$ of $R$ is called an \emph{$S$-irreducible ideal} of $R$ if
	$S\cap Q=\varnothing$ and whenever
	$s(I\cap J)\subseteq Q\subseteq I\cap J$
	for some $s\in S$ and for some ideals $I,J$ of $R$,
	there exists $s'\in S$ such that
	$ss'I\subseteq Q$ or $ss'J\subseteq Q$.
\end{definition}

\begin{definition}
	Let $S$ be a multiplicatively closed subset of a $c$-lattice $L$. An element $q$ of $L$ such that $t\nleq q$ for all $t\in S$ is called an \textit{$S$-irreducible element}, if whenever $s(a\wedge b)\leq q \leq a\wedge b$ for some $s\in S$ and for some $a,b\in L$, then there exists an element $s^{\prime}\in S$ such that either $ss^{\prime}a\leq q$ or $ss^{\prime}b\leq q$.
\end{definition}

\begin{lem}\label{thm:S-irreducible-ideal-lattice}
	Let $R$ be a commutative ring with identity and let $S$ be a multiplicatively
	closed subset of $R$.
	Let $L=Id(R)$ be the $c$-lattice of all ideals of $R$ and
	$S_L=\{(s)\in L \mid s\in S\}$.
	Then an ideal $Q$ of $R$ is an $S$-irreducible ideal of $R$ if and only if
	$Q$ (viewed as an element of $L$) is an $S_L$-irreducible element of $L$.
\end{lem}

\begin{proof}
	Assume that $Q$ is an $S$-irreducible ideal of $R$.
	Then $S\cap Q=\varnothing$.
	Hence $(s)\nleq Q$ for all $(s)\in S_L$.

	Let $A,B\in L$ and suppose that
	$(s)(A\wedge B)\leq Q\leq A\wedge B$
	for some $(s)\in S_L$.
	That is,
	$s(A\cap B)\subseteq Q\subseteq A\cap B$.

	As $Q$ is an $S$-irreducible ideal, there exists $s'\in S$ such that
	$ss'A\subseteq Q$ or $ss'B\subseteq Q$.
	Equivalently,
	$(ss')A\leq Q$ or $(ss')B\leq Q$ in $L$.
	Thus $Q$ is an $S_L$-irreducible element of $L$.

	Conversely, assume that $Q$ is an $S_L$-irreducible element of $L$.
	By the definition, $(s)\nleq Q$ for all $(s)\in S_L$.
	If $a\in S\cap Q$, then $(a)\leq Q$, a contradiction.
	Hence $S\cap Q=\varnothing$.

	Let $I,J$ be ideals of $R$ such that
	$s(I\cap J)\subseteq Q\subseteq I\cap J$
	for some $s\in S$.
	Equivalently,
	$(s)(I\wedge J)\leq Q\leq I\wedge J$ in $L$.
	Since $Q$ is $S_L$-irreducible, there exists $(s')\in S_L$ such that
	$(ss')I\leq Q$ or $(ss')J\leq Q$.
	That is,
	$ss'I\subseteq Q$ or $ss'J\subseteq Q$.
	Hence $Q$ is an $S$-irreducible ideal of $R$.
\end{proof}

\begin{prop}\label{pro1}
	Let $L$ be a $c$-lattice. Let $q_1,q_2,\dots,q_n\in L$, and $p$ be an $S$-prime element of $L$ such that $\displaystyle \bigwedge_{i=1}^{n} q_i\leq p$. Then there exists $s\in S$ such that $sq_k\leq p$ for some $k$. In particular, if $\displaystyle p=\bigwedge_{i=1}^{n} q_i$, then there exists $s\in S$ such that $sq_k\leq p\leq q_k$ for some $k$.
\end{prop}

\begin{proof}
Let $s_0\in S$ be the element in the definition of the $S$-prime element $p$. We claim that, if $a_1a_2\cdots a_n\leq p$, then $s_0^{\,n-1}a_k\leq p$ for some $k$. For $n=1$ this is clear. If $n\geq 2$, then $s_0a_1\leq p$ or $s_0a_2\cdots a_n=(s_0a_2)a_3\cdots a_n\leq p$, and the induction hypothesis applied to the latter product of $n-1$ factors gives the claim. Since $\prod_{i=1}^{n} q_i \leq \bigwedge_{i=1}^{n} q_i\leq p$, we get $sq_k\leq p$ for some $k$, with $s=s_0^{\,n-1}\in S$.
	If $\displaystyle p=\bigwedge_{i=1}^{n} q_i$, then $sq_k\leq p\leq q_k$ for some $k$ and $s\in S$.
\end{proof}

\begin{thm}\label{irred}
	Let $L$ be an $S$-Noetherian lattice. Then every $S$-irreducible element of $L$ is $S$-primary.
\end{thm}
\begin{proof}
	Let $q$ be  an $S$-irreducible element of $L$. Let $a, m \in L$ such that $am\leq q$ with $sa\nleq q$ for all $s\in S$ and $m$ is a principal element. Then consider the chain $((a\vee q): m)\leq ((a\vee q):m^2) \leq \dots \leq ((a\vee q):m^k) \leq \dots $. Since $L$ is an $S$-Noetherian lattice, by Theorem \ref{5.11a.}, there exist $s\in S$ and $k\in \mathbb{N}$ such that $s((a\vee q):m^{k+1})\leq ((a\vee q):m^k)$.

	Let $c=a\vee q$ and $d=m^{k+1}\vee q$. Then $(c\wedge d):m^{k+1}=(c: m^{k+1})\wedge (d: m^{k+1})=((a\vee q):m^{k+1})\wedge((m^{k+1}\vee q):m^{k+1})=((a\vee q):m^{k+1})\wedge 1=(a\vee q):m^{k+1}$.

	Therefore for this $s$, $s((c\wedge d):m^{k+1})=s((a\vee q):m^{k+1})\leq (a\vee q):m^{k}$. Since $m^k$ and $m^{k+1}$ are  principal elements (products of principal elements are principal \cite{A}), $((c\wedge d):m^{k+1})m^{k+1}=c\wedge d\wedge m^{k+1}$ and $((a\vee q):m^{k})m^k=(a\vee q)\wedge m^k$.

	Clearly,  $d \wedge c= d\wedge (c\wedge d)=(q\vee m^{k+1})\wedge (c\wedge d)$. Since $L$ is an $r$-lattice, $L$ is  modular and  $q \leq c \wedge d $, hence by modularity, we have $(c\wedge d)\wedge(m^{k+1}\vee q)= (m^{k+1}\wedge c\wedge d)\vee q=q\vee ((c\wedge d):m^{k+1})m^{k+1}$, as $m^{k+1}$ is  principal. 
	 Hence $c\wedge d=q\vee((c\wedge d):m^{k+1})m^{k+1}$, and we have  $s(c \wedge d) \leq sq\vee s((c\wedge d):m^{k+1})m^{k+1}\leq q\vee ((a\vee q):m^{k})m^{k}m=q\vee ((a\vee q)\wedge m^k)m\leq q\vee (a\vee q)m=q\vee am\vee qm\leq q\vee q \vee q =q$. Hence $s(c\wedge d)\leq q$. Thus, we have $s(c\wedge d)\leq q\leq c\wedge d$. Since $q$ is an $S$-irreducible element, there exists $s^{\prime}\in S$ such that either $ss^{\prime}c\leq q$ or $ss^{\prime}d\leq q$. Since $ss^{\prime}c=ss^{\prime}a \vee ss^{\prime}q\nleq q$, as $sa\nleq q$ for all $s \in S$, we have $ss^{\prime}d\leq q$. Then $(ss^{\prime}m)^{k+1}\leq ss^{\prime}m^{k+1}\leq ss^{\prime}(m^{k+1}\vee q)= ss^{\prime}d\leq q$. This proves $ss^{\prime}m\leq \sqrt{q}$, since $ss'm$ is compact ($m$ is principal, hence compact, by Remark \ref{remark}).
	 
	 Since $q \leq q_S$ and $\sqrt{q} \leq (\sqrt{q})_S$,  by Lemma \ref{lem:A}, choose $v_1,v_2\in S$ with $v_1q_S\leq q$ and $v_2(\sqrt q)_S\leq\sqrt q$, and put $v=v_1v_2\in S$. Let $a,b\in L$ with $ab\leq q$. If $sa\leq q$ for some $s\in S$, then $a\leq q_S$ and $va\leq v_1q_S\leq q$. Otherwise, $sa\nleq q$ for all $s\in S$. Since $L$ is an $r$-lattice, $b=\bigvee_\alpha m_\alpha$ with principal elements $m_\alpha$; then $am_\alpha\leq q$ for each $\alpha$, so by the first part there is $s_\alpha\in S$ with $s_\alpha m_\alpha\leq\sqrt q$, i.e., $m_\alpha\leq(\sqrt q)_S$. Hence $b\leq(\sqrt q)_S$ and $vb\leq v_2(\sqrt q)_S\leq\sqrt q$. Thus, for all $a,b\in L$ with $ab\leq q$, we have $va\leq q$ or $vb\leq\sqrt q$, and $q$ is an $S$-primary element of $L$.
\end{proof}

\begin{definition}\label{def:satirr}
An $S$-saturated element $q\neq 1$ of $L$ is called \textit{saturated-irreducible} if, whenever $q=A\wedge D$ with $A,D$ $S$-saturated, then $q=A$ or $q=D$.
\end{definition}

\begin{lem}\label{lem:D}
Let $L$ be an $S$-Noetherian lattice. Then every $S$-saturated element $q\neq 1$ of $L$ is a finite meet of saturated-irreducible elements.
\end{lem}
\begin{proof}
Let $\Sigma$ be the set of $S$-saturated elements $q\neq 1$ that are not finite meets of saturated-irreducible elements, and suppose $\Sigma\neq\emptyset$. By Theorem \ref{5.11a.}, $L$ satisfies the $S$-stationary property, so by Lemma \ref{lem:B}, $\Sigma$ has a maximal element $m$.

Then $m$ is not saturated-irreducible, so $m=A\wedge D$ with $A,D$ $S$-saturated, $m<A$ and $m<D$. If $A=1$, then $m=D$, which is impossible; so $A\neq 1$, and similarly $D\neq 1$. By maximality of $m$, $A,D\notin\Sigma$, so $A$ and $D$, and hence $m$, are finite meets of saturated-irreducible elements, a contradiction.
\end{proof}

\begin{lem}\label{lem:E}
Let $L$ be an $S$-Noetherian lattice. Then every saturated-irreducible element $q$ of $L$ is primary.
\end{lem}
\begin{proof}
Let $c,d\in L$ with $cd\leq q$ and $c\nleq q$, where $q$ is saturated-irreducible. Since $d$ is a join of principal elements, it suffices to show $m\leq\sqrt q$ for every principal $m$ with $cm\leq q$. Put $A=(c\vee q)_S$. By Lemma \ref{lem:satprops}(5), $(A:m)\leq(A:m^2)\leq\cdots$ is a chain of $S$-saturated elements, so by Lemma \ref{lem:B}, there is $k$ with $(A:m^{k+1})=(A:m^k)$. Put $D=(m^{k+1}\vee q)_S$. 

Clearly, $q\leq A\wedge D$. 

Let $z\leq A\wedge D$ be compact. By Lemma \ref{lem:satprops}(2), $sz\leq m^{k+1}\vee q$ for some $s\in S$, and also $sz\leq A$. Since $q\leq A$ and $L$ is modular, $sz\leq (m^{k+1}\vee q)\wedge A=q\vee(m^{k+1}\wedge A)$. As $m^{k+1}$ is meet principal, $m^{k+1}\wedge A=m^{k+1}(A:m^{k+1})=m\cdot m^k(A:m^k)\leq mA$. Moreover, we claim that  $mA\leq q$. For this,  if $y\leq A$ is compact, then $ty\leq c\vee q$ for some $t\in S$, so $tmy\leq cm\vee qm\leq q$, and $my\leq q_S=q$. 
As $A$ is the join of the compact elements below it, $mA=\bigvee my\leq q$. Therefore $sz\leq q\vee(m^{k+1}\wedge A)\leq q\vee mA=q$, and so $z\leq q_S=q$ by Lemma \ref{lem:satprops}(1).

As $L$ is compactly generated, $A\wedge D=q$.  Since $q$ is saturated-irreducible, $q=A$ or $q=D$. But $q=A$ would give $c\leq q$. Hence $m^{k+1} \leq q=D $, and $m\leq\sqrt q$, as $m$ is compact.
\end{proof}
In the following results, we give the relation between saturated-irreducible and $S$-irreducible. 

\begin{prop}\label{prop:Sirr-satirr}
	Let $L$ be a $c$-lattice.
	\begin{enumerate}
		\item Every saturated-irreducible element of $L$ is $S$-irreducible.
		\item If $L$ is $S$-Noetherian and $q\in L$ satisfies $t\nleq q$ for all $t\in S$, then $q$ is $S$-irreducible if and only if $q_S$ is saturated-irreducible. In particular, an $S$-saturated element of an $S$-Noetherian lattice is $S$-irreducible if and only if it is saturated-irreducible.
	\end{enumerate}
\end{prop}

\begin{proof}
	(1) Let $q$ be saturated-irreducible. Since $q=q_S\neq 1$, Lemma \ref{lem:satprops}(6) gives $t\nleq q$ for all $t\in S$. Let $s(a\wedge b)\leq q\leq a\wedge b$ for some $s\in S$ and $a,b\in L$. Then $a\wedge b\leq q_S=q$ by Lemma \ref{lem:satprops}(1), so $q=a\wedge b$. By Lemma \ref{lem:satprops}(4), $q=q_S=(a\wedge b)_S=a_S\wedge b_S$, where $a_S$ and $b_S$ are $S$-saturated by Lemma \ref{lem:satprops}(3). Hence $q=a_S$ or $q=b_S$. If $q=a_S$, then $a\leq a_S=q$, so $s\cdot 1\cdot a\leq q$; similarly if $q=b_S$. Thus $q$ is $S$-irreducible (with $s'=1$).
	
	(2) Since $t\nleq q$ for all $t\in S$, we have $q_S\neq 1$ by Lemma \ref{lem:satprops}(6). By Lemma \ref{lem:A}, there is $v\in S$ with $v\,q_S\leq q$.
	
	Suppose that $q$ is $S$-irreducible, and let $q_S=A\wedge D$ with $A,D$ $S$-saturated. Then $v(A\wedge D)=v\,q_S\leq q\leq q_S=A\wedge D$, so there is $s'\in S$ with $vs'A\leq q$ or $vs'D\leq q$. In the first case $A\leq q_S$ by Lemma \ref{lem:satprops}(1), and since $q_S\leq A$, we get $q_S=A$; similarly, in the second case $q_S=D$. Hence $q_S$ is saturated-irreducible.
	
	Conversely, suppose that $q_S$ is saturated-irreducible, and let $s(a\wedge b)\leq q\leq a\wedge b$ for some $s\in S$. Then $q_S=(a\wedge b)_S=a_S\wedge b_S$ by Lemma \ref{lem:satprops}(1),(4), so $q_S=a_S$ or $q_S=b_S$. If $q_S=a_S$, then $a\leq a_S=q_S$ by Lemma \ref{lem:satprops}(1), so $va\leq v\,q_S\leq q$ by Lemma \ref{lem:A}, and hence $s(va)\leq va\leq q$. Taking $s'=v\in S$, we get $ss'a\leq q$. Similarly, if $q_S=b_S$, then $ss'b\leq q$ with $s'=v$. Thus $q$ is $S$-irreducible.
\end{proof}

\begin{exa}
	An $S$-irreducible element need not be $S$-saturated. Let $L=Id(\mathbb{Z})$, $S=\{(n)\mid n\in\mathbb{Z}\setminus 3\mathbb{Z}\}$ and $q=(6)$. By \cite[Example 2.3]{SAK} and Lemma \ref{thm:S-irreducible-ideal-lattice}, $q$ is $S$-irreducible, but $q_S=(3)\neq q$, since $(2)(3)\leq(6)$ with $(2)\in S$. In accordance with Proposition \ref{prop:Sirr-satirr}(2), $q_S=(3)$ is prime, hence saturated-irreducible.
\end{exa}

\begin{rem}
	Proposition \ref{prop:Sirr-satirr} gives another proof of Theorem \ref{irred}. If $q$ is an $S$-irreducible element of an $S$-Noetherian lattice $L$, then $q_S$ is saturated-irreducible by Proposition \ref{prop:Sirr-satirr}(2), hence primary by Lemma \ref{lem:E}; moreover $t\nleq q_S$ for all $t\in S$ by Lemma \ref{lem:satprops}(6), and $v\,q_S\leq q\leq q_S$ for some $v\in S$ by Lemma \ref{lem:A}. Therefore $q$ is $S$-primary by Lemma \ref{lem:C}. Thus $S$-irreducibility is the ``up to a factor from $S$'' form of saturated-irreducibility, and the existence of $S$-primary decomposition (Theorem \ref{Existence}) is obtained by working exactly with the saturated elements.
\end{rem}

\begin{thm}[\textbf{Existence of $S$-Primary Decomposition}]\label{Existence}
	Let $L$ be an $S$-Noetherian lattice in which every element of $S$ is meet principal. Then every element $a$ of $L$ with $t\nleq a$ for all $t\in S$ (equivalently, $a_S\neq 1$) can be written as a finite meet of $S$-primary elements.
\end{thm}
\begin{proof}
Let $a\in L$ with $t\nleq a$ for all $t\in S$. By Lemma \ref{lem:satprops}(3),(6), $a_S$ is $S$-saturated and $a_S\neq 1$. By Lemma \ref{lem:D}, $a_S=Q_1\wedge Q_2\wedge\cdots\wedge Q_n$ $(n\geq 1)$, where each $Q_i$ is saturated-irreducible, that is, $Q_i$ is $S$-saturated and $Q_i\neq 1$; by Lemma \ref{lem:E}, each $Q_i$ is primary. Moreover, $t\nleq Q_i$ for all $t\in S$, since $(Q_i)_S=Q_i\neq 1$,  by Lemma \ref{lem:satprops}(6).

By Lemma \ref{lem:A}, there is $v\in S$ with $v\,a_S\leq a$. We claim $(a_S:v)=a_S$. Indeed, if $x$ is compact and $x\leq (a_S:v)$, then $sxv\leq a$ for some $s\in S$ (Lemma \ref{lem:satprops}(2)), so $x\leq a_S$; as $L$ is compactly generated, $(a_S:v)\leq a_S$, and the reverse inequality is clear. Since $v$ is meet principal, $a_S\wedge v=v((a_S:v)\wedge 1)=v\,a_S\leq a$. As $a\leq a_S$ and $L$ is modular, $
a_S\wedge(a\vee v)=a\vee(a_S\wedge v)=a$.
Therefore $a=\bigwedge_{i=1}^n q_i$, where $q_i=Q_i\wedge(a\vee v)$. Since $vQ_i\leq Q_i$ and $vQ_i\leq v\leq a\vee v$, we have $vQ_i\leq q_i\leq Q_i$, and by Lemma \ref{lem:C} each $q_i$ is $S$-primary.
\end{proof}

\begin{rem}
\begin{enumerate}
\item The hypothesis $t\nleq a$ for all $t\in S$ cannot be omitted. In $L=Id(\mathbb{Z}_{12})$ with $S=\{(1),(3)\}$ ($L$ is finite, hence $S$-Noetherian, and $(3)$ is principal), the proper element $(3)$ is not a finite meet of $S$-primary elements, since every $S$-primary element $q$ satisfies $(3)\nleq q$.
\item For $L=Id(R)$, the elements of $S_L$ are principal ideals, which are principal elements of $L$. Hence Theorem \ref{Existence} applies to every $S$-Noetherian ring and yields an $S$-primary decomposition of every ideal of $R$ disjoint from $S$.
\end{enumerate}
\end{rem}

\begin{rem}[On the proof of {\cite[Theorem 2.2]{SAK}}]\label{rem:SAKgap}
The statement of \cite[Theorem 2.2]{SAK} (every proper ideal of an $S$-Noetherian ring disjoint from $S$ is a finite intersection of $S$-primary ideals) is true, but its proof via $S$-irreducible ideals contains a gap, and therefore cannot simply be transferred to lattices.

\begin{enumerate}
\item In that proof, $I$ is an $S$-maximal element of the family $E$ of ideals disjoint from $S$ without an $S$-primary decomposition. Since $I$ is not $S$-irreducible, it is not irreducible, so $I=J\cap K$ with $I\neq J$ and $I\neq K$. The proof then asserts: ``Since $I$ is not $S$-irreducible, $sJ\not\subseteq I$ and $sK\not\subseteq I$ for all $s\in S$.'' This does not follow. That $I$ is not $S$-irreducible only gives \emph{some} ideals $J',K'$ and $s_0\in S$ with $s_0(J'\cap K')\subseteq I\subseteq J'\cap K'$ and $s_0s'J'\not\subseteq I$, $s_0s'K'\not\subseteq I$ for all $s'\in S$; it says nothing about the particular pair $J,K$ with $I=J\cap K$. For an exact decomposition $I=J\cap K$ one may well have $sK\subseteq I$ for some $s\in S$: in \cite[Example 2.3]{SAK}, $R=\mathbb{Z}$, $S=\mathbb{Z}\setminus 3\mathbb{Z}$, $6\mathbb{Z}=2\mathbb{Z}\cap 3\mathbb{Z}$ with $6\mathbb{Z}\neq 2\mathbb{Z},3\mathbb{Z}$, and yet $2\cdot 3\mathbb{Z}\subseteq 6\mathbb{Z}$ with $2\in S$. If such a $K$ lies in $E$, the $S$-maximality of $I$ yields no contradiction.

\item Choosing instead the witnesses $J',K'$ gives $J',K'\notin E$, but only $s_0(J'\cap K')\subseteq I\subseteq J'\cap K'$, so a decomposition of $J'\cap K'$ does not give one of $I$. The proof needs both exactness and the $S$-condition for the same pair, and neither choice provides both.

\item Moreover, $J\notin E$ means that $J$ is $S$-primary decomposable \emph{or} $J\cap S\neq\varnothing$; the second case is not addressed (it can be handled, since at most one of $J,K$ meets $S$).

\item Similarly, in \cite[Theorem 2.1]{SAK}  the element $t=ss'$ depends on $a$ and $b$, whereas the definition of an $S$-primary ideal requires one $s$ for all pairs; this is repaired as in the corrected proof of Theorem \ref{irred}, using Lemma \ref{lem:A}.
\end{enumerate}

A complete proof of \cite[Theorem 2.2]{SAK} follows from the idea of \cite[Remark 2.1]{SAK}. Let $I\cap S=\varnothing$ and let $S(I)=\{x\in R\mid sx\in I \text{ for some } s\in S\}$. Since $R$ is $S$-Noetherian, $S(I)$ is $S$-finite, so $sS(I)\subseteq I$ for some $s\in S$, i.e., $(I:s)=S(I)$. As $S^{-1}R$ is Noetherian, contracting a primary decomposition of $S^{-1}I$ gives $S(I)=Q_1\cap\cdots\cap Q_n$ with each $Q_i$ primary and $Q_i\cap S=\varnothing$. Then $I=S(I)\cap(I+Rs)$ (by \cite[Remark 2.1]{SAK}), so $I=\bigcap_{i=1}^n\bigl(Q_i\cap(I+Rs)\bigr)$, and each $Q_i\cap(I+Rs)$ is $S$-primary by \cite[Proposition 2.1(2)]{SAK}, as $s\in I+Rs$.

The proof of Theorem \ref{Existence} is the lattice version of this argument: Lemma \ref{lem:A} replaces $(I:s)=S(I)$; Lemmas \ref{lem:B}, \ref{lem:D} and \ref{lem:E} replace the primary decomposition of $S(I)$ obtained from the Noetherian ring $S^{-1}R$ (localization is not available in a lattice); the identity $a=a_S\wedge(a\vee v)$ replaces $I=S(I)\cap(I+Rs)$; and Lemma \ref{lem:C} replaces \cite[Proposition 2.1(2)]{SAK}. The identity $I=S(I)\cap(I+Rs)$ uses $Rs\cap X=s(X:s)$, i.e., that the principal ideal $Rs$ is meet principal, which holds automatically in $Id(R)$; this is why Theorem \ref{Existence} assumes that every element of $S$ is meet principal.
\end{rem}

\begin{cor}\label{}
	Let $L$ be an $S$-Noetherian lattice in which every element of  $S$ is meet principal. Then every radical element $a$ of $L$ such that $t\nleq a$ for all $t\in S$ is the meet of finitely many $S$-prime elements.
\end{cor}
\begin{proof}
	By Theorem \ref{Existence}, there exist finitely many $S$-primary elements $q_1,q_2,\dots,q_n$ such that $a=q_1\wedge q_2\wedge \dots \wedge q_n$. Also $\sqrt{q_i}=p_i$ is $S$-prime for $i=1,2,\dots, n$, by Lemma \ref{lem:radSprime}. Hence $a=\sqrt{a}=\sqrt{q_1\wedge q_2\wedge \dots \wedge q_n}=\sqrt{q_1}\wedge \sqrt{q_2}\wedge \dots \sqrt{q_n}=p_1\wedge p_2\wedge \dots \wedge p_n$.
\end{proof}

To prove the uniqueness theorem, we need the following results.
The $S$-saturation $a_S$ was defined in Definition \ref{def:sat}; its properties are listed in Lemma \ref{lem:satprops}.

\begin{prop}\label{ps-prime}
	Let $L$ be a $c$-lattice and let $S \subseteq L$ be a multiplicatively closed subset. If $p \in L$ is an $S$-prime element, then its $S$-saturation $p_S$ is a prime element of $L$.
\end{prop}
\begin{proof}
	First, we claim that $p_S\neq 1$, otherwise $t\leq p$ for some $t\in S$ by Lemma \ref{lem:satprops}(6), which is impossible. Let $a,b\in L_{\star}$ with $ab\leq p_S$. Then (as $ab$ is compact) $sab\leq p$ for some $s\in S$. Since $p$ is $S$-prime, there exists $s'\in S$ such that $s'sa\leq p$ or $s'b\leq p$. Therefore $a\leq p_S$ or $b\leq p_S$. Thus $p_S$ is a prime element of $L$.
\end{proof}

 \begin{lem}\label{radical=ps}
	Let $L$ be a $c$-lattice and $S \subseteq L$ be a multiplicatively closed subset. Let $q$ be an $S$-$p$-primary element of $L$ and $x\in L$ such that $sx\nleq q$ for all $s\in S$. Then $(q:sx)$ is an $S$-primary element of $L$, and there exists $t\in S$ such that $t(\sqrt{(q:sx)})\leq p\leq \sqrt{(q:sx)}$ for all $s\in S$. In particular, if $sx \not\leq q$ for all $s \in S$ and $x\in L$, then
	$(\sqrt{(q : x)})_S = p_S$.
\end{lem}

\begin{proof}
	First we prove that $(q:sx)$ is an $S$-primary element for all $s\in S$. Let $s\in S$. If $s^{\prime}\leq (q:sx)$ for some $s^{\prime}\in S$, then $s^{\prime}sx\leq q$, a contradiction to $sx\nleq q$ for all $s\in S$. Therefore $s^{\prime}\nleq (q:sx)$ for all $s^{\prime}\in S$.

	Let $a, ~b \in L$ such that  $ab\leq (q:sx)$ with $s^{\prime}a\nleq \sqrt{(q:sx)}$ for all $s^{\prime}\in S$. Then $(sxb)a\leq q$. Since $q$ is an $S$-primary element, there exists $s^{\prime\prime}\in S$ such that $s^{\prime\prime}sxb\leq q$ or $s^{\prime\prime}a\leq \sqrt{q}$. If $s^{\prime\prime}a\leq \sqrt{q}$, then $s^{\prime\prime}a\leq \sqrt{(q:sx)}$, a contradiction. Therefore $sxs^{\prime\prime}b \leq q$. Hence $s^{\prime\prime}b\leq (q:sx)$. Thus $(q:sx)$ is $S$-primary element of $L$, with the same element $s''$ as for $q$.

	Now, let $w$ be any compact element of $ L$ such that  $w\leq \sqrt{(q:sx)}$. This implies $w^n\leq (q:sx)$ for some positive integer $n$. This implies $w^nsx\leq q$. Since $q$ is $S$-primary element and $sx\nleq q$ for all $s\in S$, there exists $t\in S$ such that $tw^n\leq \sqrt{q}$. Hence $(tw^n)^m\leq q$ for some positive integer $m$. This gives $(tw)^{mn}\leq (tw^n)^m\leq q$. Therefore $tw\leq \sqrt{q}$, i.e.,  $w \leq (\sqrt{q} : t)$. Since each compact element $\leq \sqrt{(q:sx)} $ is also $\leq (\sqrt{q} : t) $.  This gives $ \sqrt{(q:sx)} \leq (\sqrt{q} : t) $, i.e., $t(\sqrt{(q:sx)})\leq \sqrt{q}=p$. Also $q\leq (q:sx)$ implies $p=\sqrt{q}\leq \sqrt{(q:sx)}$. Thus, $t(\sqrt{(q:sx)})\leq p\leq \sqrt{(q:sx)}$ for all $s\in S$.

	In particular, for $s=1$, $t(\sqrt{(q:x)})\leq p\leq \sqrt{(q:x)}$ for some $t\in S$. This implies $(\sqrt{(q : x)})_S = p_S$, because $p\leq y$ gives $p_S\leq y_S$, while $ty\leq p$ gives $y\leq p_S$ and hence $y_S\leq p_S$, where $y=\sqrt{(q:x)}$.
\end{proof}

\begin{thm}[First $S$-Uniqueness Theorem]
	Let $L$ be a $c$-lattice and let $S \subseteq L$ be a multiplicatively closed subset.  Let $a \in L$ and let $ a=\bigwedge_{i=1}^{n} q_i$ be a minimal $S$-primary decomposition of $ a,$  
	where each $q_i$ is an  $S$-$p_i$-primary element of $L$ and $p_i = \sqrt{q_i}$. Then the set of $S$-saturated prime elements $\{ (p_1)_S, (p_2)_S, \dots, (p_n)_S \}$ is uniquely determined by $a$ and is precisely the set \linebreak $	\mathbb{P} = \left\{\, (\sqrt{(a:x)})_S \;\middle|\; x \in L,\ sx \nleq a \ \text{for all } s \in S,\ \text{and } (\sqrt{(a:x)})_S \text{ is prime} \right\}$. In particular, the set $\{ (p_i)_S \}$ is independent of the chosen minimal $S$-primary decomposition of $a$.
\end{thm}
\begin{proof}
	We prove that the set of $S$-saturated primes $\{(p_1)_S,\dots,(p_n)_S\}$ coincides with $\mathbb{P}$.

	Fix $j \in \{1,\dots,n\}$. By the minimality of the decomposition, $\bigwedge_{i \ne j} q_i \nleq (q_j)_S$. Hence there exists $x \in L$  such that $x \le \bigwedge_{i \ne j} q_i$ and $sx \nleq q_j$ for all $s \in S$. Since $a = \bigwedge_{i=1}^n q_i \le q_j$, it follows that $sx \nleq a$ for all $s \in S$.

	It is easy to observe that $(a:x) = \bigwedge_{i=1}^n (q_i:x)$. For $i \ne j$,  and as $x \le q_i$ implies $(q_i:x)=1$, hence $(a:x) = (q_j:x)$ and so we have $\sqrt{(a:x)} = \sqrt{(q_j:x)}$.

Since $sx \nleq q_j$ and $q_j$ is $S$-$p_j$-primary, we have, by Lemma \ref{radical=ps}, $(\sqrt{(q_j:x)})_S = (p_j)_S$. By Proposition \ref{ps-prime}, $(p_j)_S$ is a prime element.

	Thus  $(\sqrt{(a:x)})_S = (p_j)_S$, showing $(p_j)_S \in \mathbb{P}$.

	Now, we prove the converse.

	Since $(a : x) = \left( \bigwedge_{i=1}^n q_i : x \right) = \bigwedge_{i=1}^n (q_i : x)$, we have $\sqrt{\bigwedge_{i=1}^n (q_i : x)} = \bigwedge_{i=1}^n \sqrt{(q_i : x)}$. This further gives $(\sqrt{(a:x)})_S =\left( \sqrt{\bigwedge_{i=1}^n (q_i : x)} \right)_S = \left( \bigwedge_{i=1}^n \sqrt{q_i : x} \right)_S= \bigwedge_{i=1}^n (\sqrt{q_i : x})_S $   as $a_S \wedge b_S = (a \wedge b)_S$ (Lemma \ref{lem:satprops}(4)). Hence $(\sqrt{(a:x)})_S \leq (\sqrt{q_i : x})_S$ for all $i$.

	Let $\pi = (\sqrt{(a:x)})_S \in \mathbb{P}$, where $sx \nleq a=\bigwedge_{i=1}^n q_i$ for all $s\in S$.

	Since $\pi$ is prime and  $ \bigwedge_{i=1}^n (\sqrt{q_i : x})_S =(\sqrt{(a:x)})_S =\pi$,  there exists  some $j$ such that $(\sqrt{q_j : x})_S \leq (\sqrt{(a:x)})_S  $. This together with $(\sqrt{(a:x)})_S \leq (\sqrt{q_i : x})_S$ for all $i$ gives that $(\sqrt{(a:x)})_S = (\sqrt{q_j: x})_S$.

	 If $sx \le q_j$ for some $s \in S$, then $s \le (q_j:x)$ and hence $(\sqrt{(q_j:x)})_S = 1$, contradicting the fact that $\pi$ is a proper prime element. Therefore $sx \nleq q_j$ for all $s \in S$, and by Lemma \ref{radical=ps}, $(\sqrt{(a:x)})_S =(\sqrt{(q_j:x)})_S = (p_j)_S$.\end{proof}

\begin{lem}\label{lem:qS-primary}
	Let $L$ be a $c$-lattice and let $q$ be an $S$-primary element of $L$. Then $\sqrt{q_S}=(\sqrt q)_S$, and $q_S$ is a primary element of $L$ with radical $(\sqrt q)_S$.
\end{lem}
\begin{proof}
	Recall that if $x$ is compact and $x\leq\sqrt b$, then $x^n\leq b$ for some $n$. Let $x$ be compact. If $x\leq\sqrt{q_S}$, then $x^n\leq q_S$ for some $n$, so $ux^n\leq q$ for some $u\in S$ by Lemma \ref{lem:satprops}(2); hence $(ux)^n\leq ux^n\leq q$, so $ux\leq\sqrt q$ and $x\leq(\sqrt q)_S$. Conversely, if $x\leq(\sqrt q)_S$, then $ux\leq\sqrt q$ for some $u\in S$, so $u^nx^n=(ux)^n\leq q\leq q_S$ for some $n$, whence $x^n\leq q_S$ by Lemma \ref{lem:satprops}(1),(3), and $x\leq\sqrt{q_S}$. As $L$ is compactly generated, $\sqrt{q_S}=(\sqrt q)_S$.
	
	Since $t\nleq q$ for all $t\in S$, we have $q_S\neq 1$ by Lemma \ref{lem:satprops}(6). Let $s\in S$ be as in the definition of the $S$-primary element $q$, and let $c,d\in L$ with $cd\leq q_S$ and $c\nleq q_S$. Choose a compact $c_0\leq c$ with $c_0\nleq q_S$, and let $d_0\leq d$ be compact. Then $c_0d_0$ is compact and $c_0d_0\leq q_S$, so $(uc_0)d_0\leq q$ for some $u\in S$. Hence $s(uc_0)\leq q$ or $sd_0\leq\sqrt q$. The first case gives $c_0\leq q_S$, which is impossible; therefore $sd_0\leq\sqrt q$, i.e., $d_0\leq(\sqrt q)_S=\sqrt{q_S}$. As $d$ is the join of such $d_0$, we get $d\leq\sqrt{q_S}$. Thus $q_S$ is primary.
\end{proof}

\begin{definition}
	Let $a=\bigwedge_{i=1}^n q_i$ be a minimal $S$-primary decomposition of $a$ in a $c$-lattice $L$, with $p_i=\sqrt{q_i}$ and $P_i=(p_i)_S$. The prime element $P_k$ is called \textit{isolated} if it is minimal in $\{P_1,\dots,P_n\}$, i.e., $P_i\nleq P_k$ for all $i\neq k$.
\end{definition}

\begin{thm}[Second $S$-Uniqueness Theorem]\label{thm:second-uniq}
	Let $L$ be a $c$-lattice and let $a=\bigwedge_{i=1}^n q_i$ be a minimal $S$-primary decomposition, where $q_i$ is $S$-$p_i$-primary, and $P_i=(p_i)_S$. If $P_k$ is isolated, then
	\[
	(q_k)_S=\bigvee\{x\in L_*\mid xy\leq a \text{ for some } y\in L_* \text{ with } y\nleq P_k\}.
	\]
	Consequently, $(q_k)_S$ is uniquely determined by $a$, i.e., it is independent of the chosen minimal $S$-primary decomposition of $a$.
\end{thm}

\begin{proof}
	For an element $a$ and a prime element $P$ of a $c$-lattice $L$, put $a^{(P)}=\bigvee\{x\in L_*\mid xy\leq a \text{ for some } y\in L_* \text{ with } y\nleq P\}$. 
	
	By Lemma \ref{lem:radSprime}, each $p_i$ is $S$-prime; hence $t\nleq p_i$ for all $t\in S$, and so $P_i=(p_i)_S\neq 1$ by Lemma \ref{lem:satprops}(6). By Proposition \ref{ps-prime}, $P_i$ is a prime element. Since $(P_i)_S=P_i\neq 1$ by Lemma \ref{lem:satprops}(3), Lemma \ref{lem:satprops}(6) gives $t\nleq P_i$ for all $t\in S$.
	We claim that $a^{(P_k)}\leq(q_k)_S$. 
	
	Let $D=\{x\in L_*\mid xy\leq a \text{ for some } y\in L_* \text{ with } y\nleq P_k\}$. If $x_1,x_2\in D$ with witnesses $y_1,y_2$, then $(x_1\vee x_2)y_1y_2\leq a$ and $y_1y_2\nleq P_k$, as $P_k$ is prime; so $D$ is closed under finite joins. Hence a compact $x\leq a^{(P_k)}=\bigvee D$ lies below some $x'\in D$, and then $xy\leq x'y\leq a$ for a compact $y\nleq P_k$.

	Let $x\in L_*$ with $x \leq a^{(P_k)}$. Then there exists $y\in L_*$ with $xy\leq a$ and $y\nleq P_k$. Let $s_k\in S$ be as in the definition of the $S$-primary element $q_k$. 
	
	Since $xy\leq a\leq q_k$, the definition of the $S$-primary element $q_k$  gives $s_kx\leq q_k$ or $s_ky\leq \sqrt{q_k}=p_k$.

	 If $s_ky\leq p_k$, then $y\leq (p_k)_S=P_k$ by Lemma \ref{lem:satprops}(1), a contradiction. Hence $s_kx\leq q_k$, so $x\leq (q_k)_S$. Since $L$ is compactly generated, we have $a^{(P_k)}\leq(q_k)_S$.
	
	Further, we observe that $(q_k)_S\leq a^{(P_k)}$. For this,  let $x\leq(q_k)_S$ be compact; then $sx\leq q_k$ for some $s\in S$ by Lemma \ref{lem:satprops}(2). Let $i\neq k$. Since the decomposition is minimal, $P_i\neq P_k$, and since $P_k$ is isolated, $P_i\nleq P_k$. As $L$ is compactly generated, there is a compact $z_i\leq P_i$ with $z_i\nleq P_k$. By Lemma \ref{lem:satprops}(2), $t_iz_i\leq p_i=\sqrt{q_i}$ for some $t_i\in S$, and since $t_iz_i$ is compact, (say), $y_i=(t_iz_i)^{n_i}\leq q_i$ for some $n_i$. As $P_k$ is prime, $t_i\nleq P_k$ and $z_i\nleq P_k$, we get $y_i\nleq P_k$. Put $y=s\prod_{i\neq k}y_i$. Then $y$ is compact and $y\nleq P_k$, since $P_k$ is prime and $s\nleq P_k$. Moreover, $xy\leq sx\leq q_k$, and $xy\leq y_i\leq q_i$ for $i\neq k$. Hence $xy\leq a$, and so $x\leq a^{(P_k)}$. As $L$ is compactly generated, $(q_k)_S\leq a^{(P_k)}$.
	
	Finally, $a^{(P_k)}$ depends only on $a$ and $P_k$, and $P_k$ depends only on $a$ by the First $S$-Uniqueness Theorem. Hence $(q_k)_S$ is uniquely determined by $a$.
\end{proof}

\begin{rem}\label{rem:second-uniq}
	The isolated component $q_k$ itself need not be unique; only its $S$-saturation is. Let $L=Id(\mathbb{Z})$ and $S=\{(n)\mid \gcd(n,6)=1\}$. Then
	\[
	(30)=(10)\wedge(3)=(2)\wedge(15)
	\]
	are two minimal $S$-primary decompositions: $(10)$ and $(15)$ are $S$-prime, with $s=(5)$, and $(2)$, $(3)$ are prime. The $S$-saturated radicals are $(2)$ and $(3)$ in both cases, and both are isolated. The components belonging to $(2)$ are $(10)$ and $(2)$, which are different, while their $S$-saturations coincide: $(10)_S=(2)_S=(2)$. In particular, \cite[Theorem 2.5]{SAK}, which asserts that the intersection of the isolated components is independent of the decomposition, does not hold as stated; its proof only shows that this intersection lies between $tS'(I)$ and $S'(I)$ for some $t\in S$.
\end{rem}

\noindent\textbf{Acknowledgements:} The authors would like to thank Professor Akiyoshi Sannai of Kyoto University for pointing out a gap in the implication $(3) \implies (4)$ of earlier version of  Theorem \ref{5.11a.} using AI Math Referee, a mathematical auditing AI tool being developed by him. Further, the authors would like to  thank the referee for the valuable suggestions which improved the presentation of the paper.

	\noindent\textbf{Conflict of interest:} The authors declare that there is no conflict of
	interest regarding the publishing of this paper.

	\noindent\textbf{Authorship Contributions :} All the authors contributed to the study on $S$-Noetherian Lattices.  Also, all the  authors read and approved the final version of the manuscript.

	\noindent \textbf{Data Availability Statement :} Data sharing does not apply to this article, as no datasets were generated or analyzed during the current study.


\nocite{*}
\begin{thebibliography}{99}
\bibitem{HH} A. Hamed and S. Hizem, {\it Modules Satisfying the $S$-Noetherian property and $S$-ACCR}, Comm.  Algebra, {\bf 44 (5)} (2016), 1941-1951.

		\bibitem{AAJ} F. Alarcon, D. D. Anderson and C. Jayaram, {\it Some results on abstract commutative ideal theory}, Period. Math. Hung., {\bf 30} (1995), 1-26.

		\bibitem{A} D. D. Anderson, {\it Abstract commutative ideal theory without chain condition}, Algebra Universal., {\bf 6} (1976), 131-145.

		\bibitem{AT} D. D. Anderson and T. Dumitrescu, {\it $S$-Noetherian Rings}, Comm. Algebra, {\bf 30 (9)} (2002), 4407-4416.

		\bibitem{BRT} Z. Bilgin, M. L. Reyes and \"U. Tekir, {\it On right $S$-Noetherian rings and $S$-Noetherian modules}, Comm. Algebra, {\bf 46} (2018), 863-869.

		\bibitem{D} R. P. Dilworth, {\it Abstract commutative ideal theory}, Pacific J. Math., {\bf 12} (1962), 481-498.

		\bibitem{HM} A. Hamed and A. Malek, {\it $S$-prime ideals of a commutative ring}, Beitr. Algebra Geom., {\bf 61} (2020),533-542.

		\bibitem{JS} V.  Joshi and S. Sarode, {\it Beck's conjecture and multiplicative
			lattices}, Discrete Math., {\bf 338 (3)} (2015), 93-98.

		\bibitem{JG} C. Jayaram, {\it Laskerian Lattices}, Czech. Math. J., {\bf 53 (128)} (2003), 351-363.

		 \bibitem{kap} I. Kaplansky,  \textit{Commutative Rings}, The University of Chicago Press: Chicago and London, 1974.

		\bibitem{M} E. Massaoud, {\it $S$-primary ideals of a commutative ring}, Comm. Algebra, {\bf 50 (3)} (2022), 988–997.

		\bibitem{SJ} S. Sarode and V. Joshi, {\it $\mathfrak {X} $-elements in multiplicative lattices-A generalization of $ J $-ideals, $ n $-ideals and $ r $-ideals in rings}, Int. Electron. J. Algebra, {\bf 32 (32)} (2022), 46-61.

		\bibitem{SJP} S. Sarode, C. Patil and V. Joshi, {\it On $S$-prime element principle} (	arXiv:2604.20820 ).

		\bibitem{SAK} T. Singh, A. Ansari and S. Kumar, {\it A study of $S$-primary decompositions}, Czech. Math. J., {\bf 75 (125)} (2025), 1241–1253.

		\bibitem{V} S. Visweswaran, {\it Some results on $S$-primary ideals of a commutative ring}, Beitr. Algebra Geom. {\bf 63} (2022), 247-266.

		\bibitem{WD} M. Ward and R. P. Dilworth, {\it Residuated lattices}, Trans. Amer. Math. Soc.,  45 (1939), 335-354.
\end{thebibliography}
\end{document}